\documentclass{article}
\usepackage{latexsym,amssymb,amsthm,amsmath,cite,enumerate}
\newtheorem{theorem}{Theorem}[section]
\newtheorem{lemma}[theorem]{Lemma}

\begin{document}

\title{Forcing Brushes}

\author{D. Meierling \and D. Rautenbach}

\date{}

\maketitle

\begin{center}
Institut f\"{u}r Optimierung und Operations Research, Universit\"{a}t Ulm, Germany, \{\texttt{dirk.meierling,dieter.rautenbach}\}\texttt{@uni-ulm.de}\\[3mm]
\end{center}

\begin{abstract}
We give short and simple proofs of the inequalities 
$B(G)\leq Z(L(G))$ and $Z(G)\leq Z(L(G))$
first established by Erzurumluo\u{g}lu, Meagher, and Pike,
where $G$ is a graph without isolated vertices,
$B(G)$ is the brushing number of $G$,
$Z(G)$ is the zero forcing number of $G$,
and $L(G)$ is the line graph of $G$.
\end{abstract}

{\small 
\begin{tabular}{lp{13cm}}
\end{tabular}
}

\section{Introduction}

Erzurumluo\u{g}lu, Meagher, and Pike~\cite{emp} recently discovered an inequality between 
the brushing number of a graph and the zero forcing number of its line graph;
thereby linking two graph parameters whose connections have not been studied so far.
Our goal in this note is to provide short and simple proofs for two of their main results. 

Let $G$ be a non-empty, simple, finite, and undirected graph.
For an integer $k$, let $[k]$ denote the set of positive integers at most $k$.
The {\it zero forcing number} $Z(G)$ of $G$~\cite{aim} 
is the minimum positive integer $k$ 
for which there are $k$ vertices $u_1,\ldots,u_k$ of $G$, 
and a linear order $u_{k+1},\ldots,u_n$ of the remaining vertices of $G$ 
such that, for every $j$ in $[n]\setminus [k]=\{ k+1,\ldots,n\}$,
there is some $i$ in $[j-1]$
such that $u_j$ is the unique neighbor of $u_i$ in $G$ that is contained in $\{ u_j,u_{j+1},\ldots,u_n\}$;
in which case we say that $u_i$ {\it forces} $u_j$.
The {\it brushing number} $B(G)$ of $G$~\cite{bfgpp,prr}
is the minimum positive integer $k$
for which there is some acyclic orientation $\vec{G}$ of $G$,
and $k$ directed paths in $\vec{G}$ 
such that each directed edge of $\vec{G}$ 
belongs to at least one of these paths.
Let $L(G)$ denote the {\it line graph} of $G$ 
whose vertex set $V(L(G))$ is the edge set $E(G)$ of $G$,
and in which two vertices $e$ and $f$ are adjacent 
if and only if $e$ and $f$ are incident as edges of $G$.

The following are two of the main results of Erzurumluo\u{g}lu et al.~(cf.~Theorem 3.1 and 4.1 in~\cite{emp}),
and the second result actually confirmed a conjecture of Eroh, Kang, and Yi \cite{eky}.

\begin{theorem}\label{theorem1}
If $G$ is a graph without isolated vertices, then $B(G)\leq Z(L(G))$.
\end{theorem}

\begin{theorem}\label{theorem2}
If $G$ is a graph without isolated vertices, then $Z(G)\leq Z(L(G))$.
\end{theorem}

\section{Proofs Theorems \ref{theorem1} and \ref{theorem2}}

Since the brushing number and the zero forcing number are additive with respect to the components of a graph,
it suffices to consider a connected graph $G$ with $n\geq 2$ vertices and $m$ edges.
Let $Z=\{ e_1,\ldots,e_k\}$ be a zero forcing set of $L(G)$ of order $k=Z(L(G))$.
Let the linear order $e_{k+1},\ldots, e_m$ 
of the remaining vertices of $L(G)$ be as in the definition of the zero forcing number,
that is, for every $j$ in $[m]\setminus [k]$,
there is some $i$ in $[j-1]$
such that $e_i$ forces $e_j$.
It follows that there are $k$ paths $P_1,\ldots,P_k$ in $L(G)$, 
where $P_i:e_i^1\ldots e_i^{m_i}$ for each $i$ in $[k]$ is such that 
$e_i^1=e_i$, and 
$e_i^j$ forces $e_i^{j+1}$ for every $j$ in $[m_i-1]$.
The paths $P_1,\ldots,P_k$ are usually referred to as {\it forcing chains},
and it is easy to see that each $P_i$ is an induced path in $L(G)$.

We order the forcing chains such that $m_1,\ldots,m_\ell\geq 2$, and $m_{\ell+1},\ldots,m_k=1$ for some non-negative $\ell\leq k$.
Let $H$ be the subgraph of $G$ with vertex set $V(G)$, and edge set $\bigcup_{i=1}^\ell V(P_i)$.
Let $\vec{H}$ be the orientation of $H$, where the edge $e_i^j=uv$ of $G$ for some $i$ in $[k]$ and $j$ in $[m_i]$
is oriented from $u$ towards $v$, that is, $uv$ is replaced by $(u,v)$, if and only if
\begin{itemize}
\item $j\geq 2$, and the edges $e_i^{j-1}$ and $e_i^j$ of $G$ share the vertex $u$, or
\item $j\leq m_i-1$, and the edges $e_i^j$ and $e_i^{j+1}$ of $G$ share the vertex $v$.
\end{itemize}
Since each $P_i$ is an induced path in $L(G)$, the orientation $\vec{H}$ is well-defined.

The following lemma contains our key observation.
\begin{lemma}
$\vec{H}$ is acyclic.
\end{lemma}
\begin{proof}
Suppose, for a contradiction, that $\vec{H}$ contains a directed cycle $\vec{C}$.
By construction, 
some vertex $w$ of $\vec{C}$ has an outgoing directed edge $(w,x)$ 
such that the edge $wx$ does not belong to $Z$,
that is, $wx$ equals $e_t$ for some $t$ in $[m]\setminus [k]$.
Note that the directed edge $(w,x)$ is not required to belong to $\vec{C}$.
We assume that $w$ and $e_t$ are chosen such that $t$ is as small as possible.

Let $(v,w)$ be the directed edge of $\vec{C}$ entering $w$.
The choice of $w$ and $(w,x)$ imply that the edge $vw$ belongs to $Z$, that is, $vw$ equals $e_s$ for some $s$ in $[\ell]$.
The minimality of $t$ implies that $e_t$ is the successor of $e_s$ on the forcing chain starting in $e_s$, that is, $e_t=e_s^2$.
Let $(u,v)$ be the directed edge of $\vec{C}$ entering $v$.
The choice of $w$ and $(w,x)$ imply that the edge $uv$ belongs to $Z$, that is, $uv$ equals $e_r$ for some $r$ in $[\ell]\setminus \{ s\}$.
Since the successor $e_r^2$ of $e_r$ on the forcing chain starting in $e_r$ is incident --- as an edge of $G$ --- with $e_s$,
it follows that $e_r^2$ is forced before $e_t=e_s^2$,
contradicting the choice of $w$ and $(w,x)$.
This completes the proof.
\end{proof}
Let $u_1,\ldots,u_n$ be a topological ordering of $\vec{H}$.
Orienting every edge $e_i=u_ru_s$ with $i$ in $[k]\setminus [\ell]$ from $u_r$ towards $u_s$
extends the acyclic orientation $\vec{H}$ of the spanning subgraph $H$ of $G$ 
to an acyclic orientation $\vec{G}$ of $G$.
Clearly, $u_1,\ldots,u_n$ is also a topological ordering of $\vec{G}$.

\begin{proof}[Proof of Theorem \ref{theorem1}]
Since the forcing chains $P_1,\ldots,P_k$ are induced paths in $L(G)$,
there are paths $Q_1,\ldots,Q_k$ in $G$ with $E(Q_i)=V(P_i)$ for every $i$ in $[k]$.
Since each vertex of $L(G)$ belongs to some forcing chain $P_i$, 
each edge of $G$ belongs to some path $Q_i$.
By the definition of the acyclic orientation $\vec{G}$,
there are directed paths $\vec{Q}_1,\ldots,\vec{Q}_k$ in $\vec{G}$
such that $\vec{Q}_i$ is an orientation of $Q_i$ for every $i$ in $[k]$.
In view of the characterization of the brushing number 
given in the introduction, this completes the proof.
\end{proof}
The following properties of the orientation $\vec{G}$ follow immediately from its construction.
Let $u$ be a vertex of $G$.
\begin{itemize}
\item If $\vec{G}$ contains no directed edge entering $u$, then $uv\in Z$ for every directed edge $(u,v)$ in $\vec{G}$ leaving $u$.
\item If $\vec{G}$ contains at least one directed edge entering $u$, then $uv\in Z$ for all but at most one directed edge $(u,v)$ in $\vec{G}$ leaving $u$.
\end{itemize}
\begin{proof}[Proof of Theorem \ref{theorem2}]
Let $Y$ be a set of vertices of $G$ such that, for every vertex $u$ of $G$,
\begin{itemize}
\item if $\vec{G}$ contains no directed edge entering $u$, then $Y$ contains $u$ as well as all but one of its outneighbors in $\vec{G}$, and
\item if $\vec{G}$ contains at least one directed edge entering $u$, then $Y$ contains all outneighbors $v$ of $u$ in $\vec{G}$ with $uv\in Z$.
\end{itemize}
The properties of $\vec{G}$ mentioned above immediately imply that $|Y|\leq |Z|$.
If $u_1,\ldots,u_n$ is a topological ordering of $\vec{G}$, and $j$ is in $[n]$,
then either $u_j\in Y$, or there is some $i$ in $[j-1]$ such that $u_j$ is the only outneighbor of $u_i$ not in $Y$.
This easily implies that $Y$ is a zero forcing set of $G$, 
completing the proof.
\end{proof}

\end{document}